\documentclass[12pt]{amsart}

\newtheorem{theorem}{Theorem}[section]
\newtheorem{lemma}[theorem]{Lemma}
\newtheorem{corollary}[theorem]{Corollary}

\theoremstyle{plain}
\newtheorem{definition}[theorem]{Definition}
\newtheorem{example}[theorem]{Example}
\newtheorem{remark}[theorem]{Remark}

\newtheorem{question}[theorem]{Question}

\theoremstyle{definition}

\theoremstyle{remark}

\numberwithin{equation}{section}

\usepackage{fancyhdr}
\usepackage{amscd}
\usepackage{amsmath}
\usepackage{amsthm}
\usepackage{amssymb}

\begin{document}

\title[K-theory and tangent spaces to Hilbert schemes]{K-theory, local cohomology and tangent spaces to Hilbert schemes}

\author{Sen Yang}
\address{Yau Mathematical Sciences Center, Tsinghua University
\\
Beijing, China}
\email{syang@math.tsinghua.edu.cn; senyangmath@gmail.com}

\date{}

\begin{abstract}
By using K-theory, we construct a map from the tangent space to the Hilbert scheme at a point $Y$ to the local cohomology group: $\pi:  \mathrm{T_{Y}Hilb}^{p}(X) \to H_{y}^{p}(\Omega_{X/\mathbb{Q}}^{p-1})$.
We use this map $\pi$ to answer (after slight modification) a question by Mark Green and Phillip Griffiths on constructing a map from the tangent space $\mathrm{T_{Y}Hilb}^{p}(X)$ to the Hilbert scheme at a point $Y$ to the tangent space to the cycle group $TZ^{p}(X)$.
 
\end{abstract}

\maketitle

\tableofcontents

\section{\textbf{Introduction}}
\label{Introduction}
Let $X$ be a smooth projective variety over a field $k$ of characteristic $0$ and let $Y \subset X$ be a subvariety of codimension $p$.  Considering $Y$ as an element of $\mathrm{Hilb}^{p}(X)$,  it is well known that the Zariski tangent space $\mathrm{T}_{Y}\mathrm{Hilb}^{p}(X)$ can be identified with $H^{0}(Y, \mathcal{N}_{Y/X})$, where $\mathcal{N}_{Y/X}$ is the normal sheaf.

$Y$ also defines an element of the cycle group $Z^{p}(X)$ and we are interested in defining the tangent space $TZ^{p}(X)$ to the cycle group $Z^{p}(X)$. In
\cite{GGtangentspace}, Mark Green and Phillip Griffiths define $TZ^{p}(X)$ for $p=1$(divisors) and $p=\mathrm{dim}(X)$(0-cycles) and leave the general case as an open question. Much of their theory was extended by Benjamin Dribus, Jerome W. Hoffman and the author in \cite{DHY, Y-3}. In \cite{Y-3}, we define $TZ^{p}(X)$ for any integer $p$ satisfying $1\leqslant p \leqslant \mathrm{dim}(X)$, generalizing Green and Griffiths' definitions. 
We recall the following fact from \cite{Y-3} for our purpose, and refer to \cite{GGtangentspace, Y-3} for definition of $TZ^{p}(X)$.

\begin{theorem} [Theorem 2.8 in  \cite{Y-3}]  \label{theorem: Yangtangent} 
For $X$ a smooth projective variety over a field $k$ of characteristic $0$, for any integer $p \geqslant 1$, the tangent space $TZ^{p}(X)$ is identified with $\mathrm{Ker}(\partial_{1}^{p,-p})$:
\[
  TZ^{p}(X) \cong \mathrm{Ker}(\partial_{1}^{p,-p}),
\]
where $\partial_{1}^{p,-p}$ is the differential of the Cousin complex \cite{Hartshorne} of $\Omega_{X/ \mathbb{Q}}^{p-1}$ in position $p$: 
 {\footnotesize
\begin{align*}
 0 \to  \Omega_{k(X)/ \mathbb{Q}}^{p-1} \to \cdots \to 
    \bigoplus_{y \in X^{(p)}} H_{y}^{p}(\Omega_{X/\mathbb{Q}}^{p-1}) \xrightarrow{\partial_{1}^{p,-p}} \bigoplus_{x \in X^{(p+1)}} H_{x}^{p+1}(\Omega_{X/\mathbb{Q}}^{p-1}) \to \cdots.
\end{align*}
}
\end{theorem}

Now,  we want to study the relation between $\mathrm{T_{Y}Hilb}^{p}(X)$ and $TZ^{p}(X)$. The following question is suggested by Mark Green and Phillip Griffiths in \cite{GGtangentspace}(see page 18 and page 87-89):
\begin{question}  \cite{GGtangentspace} \label{question: comparetangent}
For $X$ a smooth projective variety over a field $k$ of characteristic $0$, for any integer $p \geqslant 1$, is it possible to define a map from the tangent space $\mathrm{T_{Y}Hilb}^{p}(X)$ to the Hilbert scheme at a point $Y$ to the tangent space to the cycle group $TZ^{p}(X)$:
 \[
  \mathrm{T_{Y}Hilb}^{p}(X)  \to TZ^{p}(X) ?
 \]
\end{question}

For $p=\mathrm{dim}(X)$, this has been answered affirmatively by Green and Griffiths in \cite{GGtangentspace}, see Section 7.2 for details:
\begin{theorem} \cite{GGtangentspace} \label{theorem: GGVSHIlbert}
For $p=d :=\mathrm{dim}(X)$, there exists a map from the tangent space to the Hilbert scheme at a point $Y$ to the tangent space to the cycle group
 \[
  \mathrm{T_{Y}Hilb}^{d}(X)  \to TZ^{d}(X).
 \]
\end{theorem}

The main result of this short note is to construct a map in Definition ~\ref{definition: composition}
\[
\pi: \mathrm{T_{Y}Hilb}^{p}(X) \to H_{y}^{p}(\Omega_{X/\mathbb{Q}}^{p-1}),
\]
and use this map to study the above Question ~\ref{question: comparetangent} by Mark Green and Phillip Griffiths.

In Example \ref{example: trivialexample}, we show, for general $Y\subset X$ of codimension $p$ and $Y' \in \mathrm{T_{Y}Hilb}^{p}(X)$ , $\pi(Y')$ may not lie in $TZ^{p}(X)$(the kernel of $\partial_{1}^{p,-p}$).
However, we will show that, in Theorem ~\ref{theorem: mainTheorem},
 there exists $Z \subset X$ of codimension $p$ and exists $Z' \in \mathrm{T_{Z}Hilb}^{p}(X)$ such that $\pi(Y')+\pi(Z') \in TZ^{p}(X)$.

As an application, we show how to find Milnor K-theoretic cycles in Theorem \ref{theorem: MainThmCycles}. In \cite{Y-4}, we will apply these techniques to eliminate obstructions to deforming curves on a three-fold.

\textbf{Acknowledgements}
This short note is a follow-up of \cite{Y-3}. The author is very grateful to Mark Green and Phillip Griffiths for asking interesting questions.  He is also very grateful to Spencer Bloch \cite{Bloch1} and Christophe Soul\'e(Remark 4.3) for sharing their ideas.

The author thanks the following professors for discussions: Benjamin Dribus, David Eisenbud, Jerome Hoffman, Luc Illusie, Chao Zhang. He also thanks the anonymous referee(s) for professional suggestions that improved this note a lot.

\textbf{Notations and conventions}.

(1.) K-theory used in this note will be Thomason-Trobaugh non-connective K-theory, if not stated otherwise. 

(2.) For any abelian group $M$, $M_{\mathbb{Q}}$ denotes the image of $M$ in $M \otimes_{\mathbb{Z}} \mathbb{Q}$. 

(3.) $X[\varepsilon]$ denote the first order trivial deformation of $X$, i.e., $X[\varepsilon]=X \times_{k} \mathrm{Spec}(k[\varepsilon]/(\varepsilon^{2}))$, where $k[\varepsilon]/(\varepsilon^{2})$ is the ring of dual numbers.

\section{\textbf{K-theory and tangent spaces to Hilbert schemes} }
\label{K-theory and tangent spaces to Hilbert schemes}
For $X$ a smooth projective variety over a field $k$ of characteristic $0$ and $Y \subset X$ a subvariety of codimension $p$,
let $i: Y \to X$ be the inclusion, then $i_{\ast}O_{Y}$ is a coherent $O_{X}$-module and  can be resolved by a bounded complex of vector bundles on $X$.  Let $Y^{'}$ be a first order deformation of $Y$, that is, $Y^{'} \subset X[\varepsilon] $ such that $Y^{'}$ is flat over $\mathrm{Spec}(k[\varepsilon]/(\varepsilon^{2}))$ and $Y^{'} \otimes_{k[\varepsilon]/(\varepsilon^{2})} k \cong Y$. Then $i_{\ast}O_{Y'}$ can be resolved by a bounded complex of vector bundles on $X[\varepsilon]$, where $i: Y^{'} \to X[\varepsilon]$.

 Let $D^{\mathrm{perf}}(X[\varepsilon])$ denote the derived category of perfect complexes of $O_{X}[\varepsilon]$-modules, and let $\mathcal{L}_{(i)}(X[\varepsilon]) \subset D^{\mathrm{perf}}(X[\varepsilon])$ be defined as
\[
  \mathcal{L}_{(i)}(X[\varepsilon]) := \{ E \in D^{\mathrm{perf}}(X[\varepsilon]) \mid \mathrm{codim_{Krull}(supph(E))} \geq -i \},
\]
where the closed subset $\mathrm{supph(E)} \subset X$ is the support of the total homology of the perfect complex $E$.

The resolution of $i_{\ast}O_{Y'}$, which is a perfect complex of $O_{X}[\varepsilon]$-module supported on $Y$, defines an element of the Verdier quotient $\mathcal{L}_{(-p)}(X[\varepsilon])/\mathcal{L}_{(-p-1)}(X[\varepsilon])$, denoted $[i_{\ast}O_{Y'}]$.

In general, the length of the perfect complex $[i_{\ast}O_{Y'}]$ may not be equal to $p$. Since $Y \subset X$ is of codimension $p$, for our purpose, we expect the perfect complex $[i_{\ast}O_{Y'}]$ is of length $p$. To achieve this, instead of considering $[i_{\ast}O_{Y'}]$ as an element of  the Verdier quotient $\mathcal{L}_{(-p)}(X[\varepsilon])/\mathcal{L}_{(-p-1)}(X[\varepsilon])$, we consider its image in the idempotent completion $(\mathcal{L}_{(-p)}(X[\varepsilon])/\mathcal{L}_{(-p-1)}(X[\varepsilon]))^{\#}$,  denoted $[i_{\ast}O_{Y'}]^{\#}$, where the idempotent completion is in the sense of Balmer-Schlichting \cite{B-S}.
 And we have the following result: 
\begin{theorem} \cite{B-3} \label{theorem: Balmer theorem}
 For each $i \in \mathbb{Z}$, localization induces an equivalence
\[
 (\mathcal{L}_{(i)}(X[\varepsilon])/\mathcal{L}_{(i-1)}(X[\varepsilon]))^{\#}  \simeq \bigsqcup_{x[\varepsilon] \in X[\varepsilon]^{(-i)}}D_{{x[\varepsilon]}}^{\mathrm{perf}}(X[\varepsilon])
\]
between the idempotent completion of the Verdier quotient $\mathcal{L}_{(i)}(X[\varepsilon])/\mathcal{L}_{(i-1)}(X[\varepsilon])$ and the coproduct over $x[\varepsilon] \in X[\varepsilon]^{(-i)}$ of the derived category of perfect complexes of $ O_{X[\varepsilon],x[\varepsilon]}$-modules with homology supported on the closed point $x[\varepsilon] \in \mathrm{Spec}(O_{X[\varepsilon],x[\varepsilon]})$. And consequently, one has 
\[
 K_{0}((\mathcal{L}_{(i)}(X[\varepsilon])/\mathcal{L}_{(i-1)}(X[\varepsilon]))^{\#})  \simeq \bigoplus_{x[\varepsilon] \in X[\varepsilon]^{(-i)}}K_{0}(D_{{x[\varepsilon]}}^{\mathrm{perf}}(X[\varepsilon])).
\]
\end{theorem}

Let $y$ be the generic point of $Y$ and let $\mathcal{I}_{Y}$ be the ideal sheaf of $Y$. Then there exists the following short exact sequence:
\[
 0 \to \mathcal{I}_{Y} \to O_{X} \to i_{\ast}O_{Y} \to 0, 
\]
whose localization at $y$ is the short exact sequence:
\[
0 \to (\mathcal{I}_{Y})_{y} \to O_{X, y} \to (i_{\ast}O_{Y})_{y} \to 0.
\]

We have $O_{Y,y}= O_{X, y}/(\mathcal{I}_{Y})_{y}$. Since $O_{Y,y}$ is a field, $(\mathcal{I}_{Y})_{y}$ is the maximal ideal of  the regular local ring (of dimension $p$)$O_{X, y}$. So the maximal ideal $(\mathcal{I}_{Y})_{y}$ is generated by a regular sequence of length $p$: $f_{1}, \cdots, f_{p}$. 

Let $\mathcal{I}_{Y'}$ be the ideal sheaf of $Y'$,  then $\mathcal{I}_{Y'}/(\varepsilon)\mathcal{I}_{Y'} = \mathcal{I}_{Y}$ because of flatness. So we have $(\mathcal{I}_{Y'})_{y}/(\varepsilon)(\mathcal{I}_{Y'})_{y} = (\mathcal{I}_{Y})_{y}$. Lift  $f_{1}, \cdots, f_{p}$ to $f_{1}+\varepsilon g_{1}, \cdots, f_{p}+\varepsilon g_{p}$ in $(\mathcal{I}_{Y'})_{y}$, where $g_{1}, \cdots, g_{p} \in O_{X, y}$, then $f_{1}+\varepsilon g_{1}, \cdots, f_{p}+\varepsilon g_{p}$ generates $(\mathcal{I}_{Y'})_{y}$ because of Nakayama's lemma:
\[
(\mathcal{I}_{Y'})_{y} = (f_{1}+\varepsilon g_{1}, \cdots, f_{p}+\varepsilon g_{p}). 
\]
Moreover, $f_{1}+\varepsilon g_{1}, \cdots, f_{p}+\varepsilon g_{p}$ is a regular sequence which can be checked directly.

We see that $Y$ is generically defined by a regular sequence of length $p$: $f_{1}, \cdots, f_{p}$, where $f_{1}, \cdots, f_{p} \in O_{X, y}$.  $Y'$ is generically given by lifting  $f_{1}, \cdots, f_{p}$ to $f_{1}+\varepsilon g_{1}, \cdots, f_{p}+\varepsilon g_{p}$, where $g_{1}, \cdots, g_{p} \in O_{X, y}$.
We use $F_{\bullet}(f_{1}+\varepsilon g_{1}, \cdots, f_{p}+\varepsilon g_{p})$ to denote the Koszul complex associated to the  regular sequence $f_{1}+\varepsilon g_{1}, \cdots, f_{p}+\varepsilon g_{p}$,  which is a resolution of $O_{X, y}[\varepsilon]/(f_{1}+\varepsilon g_{1}, \cdots, f_{p}+\varepsilon g_{p})$:
\[
 \begin{CD}
  0 @>>> F_{p} @>A_{p}>> F_{p-1} @>A_{p-1}>>  \dots @>A_{2}>> F_{1} @>A_{1}>> F_{0} @>>> 0,
 \end{CD}
\]
where each $F_{i}=\bigwedge^{i} (O_{X,y}[\varepsilon])^{\oplus p}$ and $A_{i}: \bigwedge^{i} (O_{X,y}[\varepsilon])^{\oplus p}  \to \bigwedge^{i-1} (O_{X,y}[\varepsilon])^{\oplus p}$ are defined as usual.

Under the equivalence in Theorem \ref{theorem: Balmer theorem}, the localization at the generic point $y$ sends $[i_{\ast}O_{Y'}]^{\#}$ to the Koszul complex $F_{\bullet}(f_{1}+\varepsilon g_{1}, \cdots, f_{p}+\varepsilon g_{p})$:
\[
[i_{\ast}O_{Y'}]^{\#} \to F_{\bullet}(f_{1}+\varepsilon g_{1}, \cdots, f_{p}+\varepsilon g_{p}).
\]

Recall that Milnor K-groups with support are rationally defined in terms of eigenspaces of Adams operations in \cite{Y-2}.
\begin{definition}  [Definition 3.2 in \cite{Y-2}] \label{definition:Milnor K-theory with support}
Let $X$ be a finite equi-dimensional noetherian scheme and $x \in X^{(j)}$. For $m \in \mathbb{Z}$, Milnor K-group with support $K_{m}^{M}(O_{X,x} \ \mathrm{on} \ x)$ is rationally defined to be 
\[
  K_{m}^{M}(O_{X,x} \ \mathrm{on} \ x) := K_{m}^{(m+j)}(O_{X,x} \ \mathrm{on} \ x)_{\mathbb{Q}},
\] 
where $K_{m}^{(m+j)}$ is the eigenspace of $\psi^{k}=k^{m+j}$ and $\psi^{k}$ are the Adams operations.

\end{definition}

\begin{theorem} [Prop 4.12 of \cite{GilletSoule}]  \label{theorem: GilletSoule}
The Adams operations $\psi^{k}$ defined on perfect complexes, defined by Gillet-Soul\'e in \cite{GilletSoule}, satisfy
$\psi^{k}(F_{\bullet}(f_{1}+\varepsilon g_{1}, \cdots, f_{p}+\varepsilon g_{p})) = k^{p}F_{\bullet}(f_{1}+\varepsilon g_{1}, \cdots, f_{p}+\varepsilon g_{p})$.
\end{theorem}

Hence, $F_{\bullet}(f_{1}+\varepsilon g_{1}, \cdots, f_{p}+\varepsilon g_{p})$ is of eigenweight $p$ and can be considered as an element of $K^{(p)}_{0}(O_{X,y}[\varepsilon] \ \mathrm{on} \ y[\varepsilon])_{\mathbb{Q}}$: 
\[
F_{\bullet}(f_{1}+\varepsilon g_{1}, \cdots, f_{p}+\varepsilon g_{p}) \in K^{(p)}_{0}(O_{X,y}[\varepsilon] \ \mathrm{on} \ y[\varepsilon])_{\mathbb{Q}}=K^{M}_{0}(O_{X,y}[\varepsilon] \ \mathrm{on} \ y[\varepsilon]).
\]

\begin{definition}  \label{definition: map1}
We define a  map $\mu:  \mathrm{T}_{Y} \mathrm{Hilb}^{p}(X) \to K^{M}_{0}(O_{X,y}[\varepsilon] \ \mathrm{on} \ y[\varepsilon])$ as follows: 
\begin{align*}
\mu:   \ \mathrm{T}_{Y} & \mathrm{Hilb}^{p}(X)  \to K^{M}_{0}(O_{X,y}[\varepsilon] \ \mathrm{on} \ y[\varepsilon]) \\
& Y' \longrightarrow  F_{\bullet}(f_{1}+\varepsilon g_{1}, \cdots, f_{p}+\varepsilon g_{p}).
\end{align*}
\end{definition}

\section{\textbf{Chern character}}
\label{Chern character}
 For any integer $m$, let $K^{(i)}_{m}(O_{X,y}[\varepsilon] \ \mathrm{on} \ y[\varepsilon],\varepsilon)_{\mathbb{Q}}$ denote the weight $i$ eigenspace of the relative K-group, that is, the kernel of the natural projection
\[
K^{(i)}_{m}(O_{X,y}[\varepsilon] \ \mathrm{on} \ y[\varepsilon])_{\mathbb{Q}}  \xrightarrow{\varepsilon =0} K^{(i)}_{m}(O_{X,y} \ \mathrm{on} \ y)_{\mathbb{Q}}.
\]

Recall that we have proved the following isomorphisms in \cite{DHY, Y-2}:
\begin{theorem} [Corollary 9.5 in \cite{DHY}, Corollary 3.11 in \cite{Y-2}] \label{theorem:Chernisomorphism}
Let $X$ be a smooth projective variety over a field $k$ of characteristic $0$ and 
let $y \in X^{(p)}$. Chern character(from K-theory to negative cyclic homology) induces the following isomorphisms between relative K-groups and local cohomology groups:
\[
 \begin{CD}
  K^{(i)}_{m}(O_{X,y}[\varepsilon] \ \mathrm{on} \ y[\varepsilon],\varepsilon)_{\mathbb{Q}} \cong H_{y}^{p}(\Omega^{\bullet,(i)}_{X/\mathbb{Q}}),
 \end{CD}
\]
where
\[
\begin{cases}
 \Omega_{X/ \mathbb{Q}}^{\bullet,(i)}  & =  \Omega^{{2i-(m+p)-1}}_{X/ \mathbb{Q}},    \mathrm{for} \  \frac{m+p}{2}  < \ i \leq m+p, \\
  \Omega_{X/ \mathbb{Q}}^{\bullet,(i)} & =  0,    \mathrm{else}.
 \end{cases}
\]

\end{theorem}
The main tool for proving these isomorphisms is the space-level versions of Goodwillie's and Cathelineau's isomorphisms, proved in the appendix B of \cite{CHW}.

Let $K^{M}_{m}(O_{X,y}[\varepsilon] \ \mathrm{on} \ y[\varepsilon],\varepsilon)$ denote the relative K-group, that is, the kernel of the natural projection
\[
K^{M}_{m}(O_{X,y}[\varepsilon] \ \mathrm{on} \ y[\varepsilon]) \xrightarrow{\varepsilon =0} K^{M}_{m}(O_{X,y} \ \mathrm{on} \ y).
\] 
In other words, $K^{M}_{m}(O_{X,y}[\varepsilon] \ \mathrm{on} \ y[\varepsilon],\varepsilon)$ is 
$K^{(m+p)}_{m}(O_{X,y}[\varepsilon] \ \mathrm{on} \ y[\varepsilon],\varepsilon)_{\mathbb{Q}}$.
In particular,  by taking $i=p$ and $m=0$ in Theorem ~\ref{theorem:Chernisomorphism}, we obtain the following formula:

\begin{corollary}
\[
 K^{M}_{0}(O_{X,y}[\varepsilon] \ \mathrm{on} \ y[\varepsilon],\varepsilon) \xrightarrow{\cong} H_{y}^{p}(\Omega^{p-1}_{X/\mathbb{Q}}).
\]
\end{corollary}

\begin{definition} \label{definition: Chern}
Let $X$ be a smooth projective variety over a field $k$ of characteristic $0$ and let $y \in X^{(p)}$. There exists the following natural surjective map
\[
\mathrm{Ch}:  K^{M}_{0}(O_{X,y}[\varepsilon] \ \mathrm{on} \ y[\varepsilon]) \to H_{y}^{p}(\Omega_{X/\mathbb{Q}}^{p-1}),
\]
which is defined to be the composition of the natural projection
\[
K^{M}_{0}(O_{X,y}[\varepsilon] \ \mathrm{on} \ y[\varepsilon]) \to 
K^{M}_{0}(O_{X,y}[\varepsilon] \ \mathrm{on} \ y[\varepsilon], \varepsilon) ,
\] 
and the following isomorphism
\[
 K^{M}_{0}(O_{X,y}[\varepsilon] \ \mathrm{on} \ y[\varepsilon],\varepsilon) \xrightarrow{\cong} H_{y}^{p}(\Omega^{p-1}_{X/\mathbb{Q}}).
\]

\end{definition}

Now, We recall a beautiful construction of Ang\'eniol and Lejeune-Jalabert which describes the map in Definition ~\ref{definition: Chern}
\[
\mathrm{Ch}:  K^{M}_{0}(O_{X,y}[\varepsilon] \ \mathrm{on} \ y[\varepsilon]) \to H_{y}^{p}(\Omega_{X/\mathbb{Q}}^{p-1}). 
\]

An element $M \in K^{M}_{0}(O_{X,y}[\varepsilon] \ \mathrm{on} \ y[\varepsilon]) \subset K_{0}(O_{X,y}[\varepsilon] \ on \ y[\varepsilon])_{\mathbb{Q}}$ is represented by a strict perfect complex $L_{\bullet}$ supported at $y[\varepsilon]$: 
\[
 \begin{CD}
  0 @>>> F_{n} @>M_{n}>> F_{n-1} @>M_{n-1}>>  \dots @>M_{2}>> F_{1} @>M_{1}>> F_{0} @>>> 0,
 \end{CD}
\]
where each $F_{i}=O_{X,y}[\varepsilon]^{r_{i}}$ and $M_{i}$'s are matrices with entries in $O_{X,y}[\varepsilon]$.

\begin{definition} [Page 24 in \cite{A-LJ}] \label{definition: local cycles}
The local fundamental class attached to this perfect complex is defined to be the following collection:
\[
 [L_{\bullet}]_{loc}=\{\frac{1}{p!}dM_{i}\circ dM_{i+1}\circ \dots \circ dM_{i+p-1}\}, i =  0, 1, \cdots
\]
where $d=d_{\mathbb{Q}}$ and each $dM_{i}$ is the matrix of absolute differentials. In other words,
\[
 dM_{i} \in \mathrm{Hom}(F_{i},F_{i-1}\otimes \Omega_{O_{X,y}[\varepsilon]/\mathbb{Q}}^{1}).
\]
\end{definition}

\begin{theorem}  [Lemme 3.1.1 on page 24, Def 3.4 on page 29 in \cite{A-LJ}] \label{theorem: cycle}
 $[L_{\bullet}]_{loc}$ defined above is a cycle in 
  $\mathcal{H}om( L_{\bullet}, \Omega_{O_{X,y}[\varepsilon]/\mathbb{Q}}^{p}\otimes L_{\bullet})$, 
  and the image of $[L_{\bullet}]_{loc}$ in 
$
 H^{p}(\mathcal{H}om( L_{\bullet}, \Omega_{O_{X,y}[\varepsilon]/\mathbb{Q}}^{p}\otimes L_{\bullet}))
$
 does not depend on the choice of the basis of $L_{\bullet}$.
 
 \end{theorem}

Since 
\[
 H^{p}(\mathcal{H}om( L_{\bullet}, \Omega_{O_{X,y}[\varepsilon]/\mathbb{Q}}^{p}\otimes L_{\bullet}))=\mathcal{E}XT^{p}(L_{\bullet}, \Omega_{O_{X,y}[\varepsilon]/\mathbb{Q}}^{p}\otimes L_{\bullet}),
\]
the above local fundamental class $[L_{\bullet}]_{loc}$ defines an element in $\mathcal{E}XT^{p}(L_{\bullet}, \Omega_{O_{X,y}[\varepsilon]/\mathbb{Q}}^{p}\otimes L_{\bullet})$: 
\[
  [L_{\bullet}]_{loc} \in \mathcal{E}XT^{p}(L_{\bullet}, \Omega_{O_{X,y}[\varepsilon]/\mathbb{Q}}^{p}\otimes L_{\bullet}).
\]

Noting $L_{\bullet}$ is supported on $y$(same underlying space as $y[\varepsilon]$), there exists the following trace map, see page 98-99 in [1] for details,
\[
 \mathrm{Tr}:  \mathcal{E}XT^{p}(L_{\bullet}, \Omega_{O_{X,y}[\varepsilon]/\mathbb{Q}}^{p}\otimes L_{\bullet}) \longrightarrow H_{y}^{p}(\Omega_{X[\varepsilon]/\mathbb{Q}}^{p}). 
\]

\begin{definition}   [Definition 2.3.2 on page 99 in \cite{A-LJ}]
The image of $[L_{\bullet}]_{loc}$ under the above trace map, denoted $\mathcal{V}^{p}_{L_{\bullet}}$, is called Newton class.
\end{definition}

$K_{0}(O_{X,y}[\varepsilon] \ on \ y[\varepsilon])$ is the Grothendieck group of the triangulated category $D^{b}(O_{X,y}[\varepsilon] \ \mathrm{on} \ y[\varepsilon])$, which is the derived category of perfect complexes of $ O_{X,y}[\varepsilon]$-modules with homology supported on the closed point $y[\varepsilon] \in \mathrm{Spec}(O_{X,y}[\varepsilon])$.
 Recall that the Grothendieck group of a triangulated category is the monoid of isomorphism objects modulo the submonoid formed from distinguished triangles.
\begin{theorem} [Proposition 4.3.1 on page 113 in \cite{A-LJ}]
The Newton class $\mathcal{V}^{p}_{L_{\bullet}}$ is well-defined on $K_{0}(O_{X,y}[\varepsilon] \ on \ y[\varepsilon])$.
\end{theorem}

The truncation map $\rfloor  \dfrac{\partial}{\partial \varepsilon}\mid_{\varepsilon =0}: \Omega_{X[\varepsilon]/\mathbb{Q}}^{p} \to \Omega_{X/\mathbb{Q}}^{p-1}$ induces a map,
\begin{align*}
\rfloor  \dfrac{\partial}{\partial \varepsilon}\mid_{\varepsilon =0} :   H_{y}^{p}(\Omega_{X[\varepsilon]/\mathbb{Q}}^{p}) \longrightarrow H_{y}^{p}(\Omega_{X/\mathbb{Q}}^{p-1}).
\end{align*}

\begin{lemma}  \label{theorem: describe}
The map(in Definition \ref{definition: Chern})
\[
\mathrm{Ch}:  K^{M}_{0}(O_{X,y}[\varepsilon] \ \mathrm{on} \ y[\varepsilon]) \to H_{y}^{p}(\Omega_{X/\mathbb{Q}}^{p-1})
\]
can be described as a composition:
{\footnotesize
\begin{align*}
 K^{M}_{0}(O_{X,y}[\varepsilon] \ & \mathrm{on} \ y[\varepsilon]) \to \mathcal{E}XT^{p}(L_{\bullet}, \Omega_{O_{X,y}[\varepsilon]/\mathbb{Q}}^{p}\otimes L_{\bullet}) \to H_{y}^{p}(\Omega_{X[\varepsilon]/\mathbb{Q}}^{p}) \to H_{y}^{p}(\Omega_{X/\mathbb{Q}}^{p-1}) \\
&  L_{\bullet} \longrightarrow  [L_{\bullet}]_{loc} \longrightarrow \mathcal{V}^{p}_{L_{\bullet}} \longrightarrow \mathcal{V}^{p}_{L_{\bullet}} \rfloor  \dfrac{\partial}{\partial \varepsilon}\mid_{\varepsilon =0} .
\end{align*}
}
\end{lemma}

In particular, for the Koszul complex $F_{\bullet}(f_{1}+\varepsilon g_{1}, \cdots, f_{p}+\varepsilon g_{p})$ in Definition ~\ref{definition: map1}, 
the Ch map can be described as follows. The following diagram,  
\[
\begin{cases}
 \begin{CD}
   F_{\bullet}(f_{1}+\varepsilon g_{1}, \cdots, f_{p}+\varepsilon g_{p}) @>>> O_{X, y}[\varepsilon]/(f_{1}+\varepsilon g_{1}, \cdots, f_{p}+\varepsilon g_{p}) \\
  F_{p}(\cong O_{X,y}[\varepsilon]) @>[F_{\bullet}]_{loc}>> F_{0} \otimes \Omega_{O_{X,y}[\varepsilon]/ \mathbb{Q}}^{p}(\cong \Omega_{O_{X,y}[\varepsilon]/ \mathbb{Q}}^{p}), 
 \end{CD}
\end{cases}
\]
where $[F_{\bullet}]_{loc}$ is short for the local fundamental class attached to  $F_{\bullet}(f_{1}+\varepsilon g_{1}, \cdots, f_{p}+\varepsilon g_{p})$, gives an element in $Ext_{O_{X,y}[\varepsilon]}^{p}(O_{X,y}[\varepsilon]/(f_{1}+\varepsilon g_{1}, \cdots, f_{p}+\varepsilon g_{p}),\Omega_{X[\varepsilon]/ \mathbb{Q}}^{p})$. This further gives an element in $H_{y}^{p}(\Omega_{X[\varepsilon]/ \mathbb{Q}}^{p} )$, denoted $\mathcal{V}^{p}_{F_{\cdot}}$. 

We use $F_{\bullet}(f_{1}, \cdots, f_{p})$ to denote the Koszul complex associated to the  regular sequence $f_{1}, \cdots, f_{p}$,  which is a resolution of $O_{X, y}/(f_{1}, \cdots, f_{p})$.
The truncation of $\mathcal{V}^{p}_{F_{\cdot}}$ in $\varepsilon$ produces an element in $H_{y}^{p}(\Omega_{X/ \mathbb{Q}}^{p-1} )$, which can be represented by the following diagram
\[
\begin{cases}
 \begin{CD}
   F_{\bullet}(f_{1}, \cdots, f_{p}) @>>> O_{X, y}/(f_{1}, \cdots,  f_{p}) \\
  F_{p}(\cong O_{X,y}) @>[F_{\bullet}]_{loc}\rfloor  \dfrac{\partial}{\partial \varepsilon}\mid_{\varepsilon =0} >> F_{0} \otimes \Omega_{O_{X,y}/ \mathbb{Q}}^{p-1}(\cong\Omega_{O_{X,y}/ \mathbb{Q}}^{p-1}).
 \end{CD}
 \end{cases}
\]
For simplicity, assuming $g_{2}= \cdots = g_{p}=0$, we see that $[F_{\bullet}]_{loc}\rfloor  \dfrac{\partial}{\partial \varepsilon}\mid_{\varepsilon =0} = g_{1}df_{2} \wedge \cdots \wedge df_{p}$ and 
the truncation of $\mathcal{V}^{p}_{F_{\cdot}}$ in $\varepsilon$ is represented by the following diagram
\[
\begin{cases}
 \begin{CD}
   F_{\bullet}(f_{1}, \cdots, f_{p}) @>>> O_{X, y}/(f_{1}, \cdots,  f_{p}) \\
  F_{p}(\cong O_{X,y}) @>g_{1}df_{2} \wedge \cdots \wedge df_{p}>> F_{0} \otimes \Omega_{O_{X,y}/ \mathbb{Q}}^{p-1}(\cong\Omega_{O_{X,y}/ \mathbb{Q}}^{p-1}).
 \end{CD}
 \end{cases}
\]

Further concrete examples can be found in Chapter 7 of \cite{GGtangentspace}(page 90-91).

\section{\textbf{The map $\pi$}}

\begin{definition} \label{definition: composition}
We define a map from $\mathrm{T_{Y}Hilb}^{p}(X)$ to $H_{y}^{p}(\Omega_{X/\mathbb{Q}}^{p-1})$ by composing Ch in $\mathrm{Definition  ~\ref{definition: Chern}}$ with $\mu$ in $\mathrm{Definition ~\ref{definition: map1}}$:
\[
\pi :  \mathrm{T_{Y}Hilb}^{p}(X) \xrightarrow{\mu}   K^{M}_{0}(O_{X,y}[\varepsilon] \ \mathrm{on} \ y[\varepsilon]) \xrightarrow{\mathrm{Ch}} H_{y}^{p}(\Omega_{X/\mathbb{Q}}^{p-1}).
\]
\end{definition}

Recall that the Cousin complex of $\Omega_{X/ \mathbb{Q}}^{p-1}$ is of the form: 
 {\footnotesize
\begin{align*}
 0 \to  \Omega_{k(X)/ \mathbb{Q}}^{p-1} \to \cdots \to 
    \bigoplus_{y \in X^{(p)}} H_{y}^{p}(\Omega_{X/\mathbb{Q}}^{p-1}) \xrightarrow{\partial_{1}^{p,-p}} \bigoplus_{x \in X^{(p+1)}} H_{x}^{p+1}(\Omega_{X/\mathbb{Q}}^{p-1}) \to \cdots
\end{align*}
}
and the tangent space $TZ^{p}(X)$ is identified with $\mathrm{Ker}(\partial_{1}^{p,-p})$, see Theorem ~\ref{theorem: Yangtangent}.

For $p=d := \mathrm{dim}(X)$, $\partial_{1}^{d,-d} = 0$ because of dimensional reason. So $TZ^{d}(X) = \mathrm{Ker}(\partial_{1}^{d,-d})=\bigoplus\limits_{y \in X^{(d)}}H_{y}^{d}(\Omega_{X/\mathbb{Q}}^{d-1})$.
\begin{corollary}
For $p=d :=\mathrm{dim}(X)$, the map $\pi $ defines a map from $\mathrm{T_{Y}Hilb}^{d}(X)$ to $TZ^{d}(X)$ and it agrees with the map by $\mathrm{Green}$ and $\mathrm{Griffiths}$ in $\mathrm{Theorem ~\ref{theorem: GGVSHIlbert}}$.
\end{corollary}

We want to know, for general $p$,  whether this map $\pi$ defines a map from $\mathrm{T_{Y}Hilb}^{p}(X)$ to $TZ^{p}(X)$, as Green and Griffiths asked in Question ~\ref{question: comparetangent}.

\begin{remark}
In an email to the author,  $\mathrm{Christophe  \ Soul\acute{e} }$ suggested that he consider the image of suitable Koszul complexes under the Ch map in Definition \ref{definition: Chern}. This leads us to the following example, showing that $\pi$ does not define a map from $\mathrm{T_{Y}Hilb}^{p}(X)$ to $TZ^{p}(X)$ in general. The Koszul technique is also used in $\mathrm{Theorem ~\ref{theorem: mainTheorem}}$.

The author sincerely thanks $\mathrm{Christophe  \ Soul\acute{e}}$ for very helpful suggestions.

\end{remark}

\begin{example}  \label{example: trivialexample}
For $X$ be a smooth projective three-fold over a field $k$ of characteristic $0$, let $Y \subset X$ be a curve with generic point $y$. We assume a point
$x \in Y \subset X$ is defined by $(f,g,h)$ and $Y$ is generically defined by $(f,g)$, then $O_{X,y}=(O_{X,x})_{(f,g)}$.

We consider the infinitesimal deformation $Y'$ of $Y$ which is generically given by $(f+ \varepsilon \dfrac{1}{h}, g)$, where $\dfrac{1}{h} \in O_{X,y}= (O_{X,x})_{(f,g)} $(Note $\dfrac{1}{h} \notin O_{X,x}$), whose Koszul complex is of the form
\[
0 \to (O_{X,x})_{(f,g)}[\varepsilon]  \xrightarrow{(g, -f - \varepsilon \dfrac{1}{h})^{\mathrm{T}}}  (O_{X,x})_{(f,g)}^{\oplus 2}[\varepsilon]  \xrightarrow{(f + \varepsilon \dfrac{1}{h},g)}  (O_{X,x})_{(f,g)}[\varepsilon] \to 0,
\]
where $(-,-)^{\mathrm{T}}$ denotes transpose.

$\pi(Y') \in H^{2}_{y}(\Omega^{1}_{X/ \mathbb{Q}})$ is represented by the following diagram, 
\[
\begin{cases}
 \begin{CD}
   (O_{X,x})_{(f,g)}  @>>> (O_{X,x})_{(f,g)}^{\oplus 2} @>>>   (O_{X,x})_{(f,g)}  @>>> (O_{X,x})_{(f,g)} / (f,g) @>>> 0  \\
   (O_{X,x})_{(f,g)}  @>\frac{1}{h} dg >> \Omega^{1}_{(O_{X,x})_{(f,g)}/ \mathbb{Q}}.
 \end{CD}
 \end{cases}
\]

Let $F_{\bullet}(f,g,h)$ be the Koszul complex of $f,g,h$:
\[
0 \to O_{X,x} \to O^{\oplus 3}_{X,x} \to O^{\oplus 3}_{X,x} \to O_{X,x} \to 0,
\]
$\partial^{2,-2}_{1}(\pi(Y'))$ in $H^{3}_{x}(\Omega^{1}_{X/ \mathbb{Q}})$ is represented by the following diagram:
\[
\begin{cases}
 \begin{CD}
   F_{\bullet}(f,g,h) @>>> O_{X,x}/(f,g,h)   \\
   O_{X,x} @>1dg >> \Omega^{1}_{O_{X,x}/ \mathbb{Q}},
 \end{CD}
 \end{cases}
\]
which is not zero.

\end{example}

This example shows that, in general, 
 the image of $\pi$ may not lie in $TZ^{p}(X)$(the kernel of $\partial_{1}^{p,-p}$). However, we will show that, in Theorem ~\ref{theorem: mainTheorem} below,
 given $Y \subset X$ of codimension $p$ and $Y' \in \mathrm{T_{Y}Hilb}^{p}(X)$,
 there exists $Z \subset X$ of codimension $p$ and $Z' \in \mathrm{T_{Z}Hilb}^{p}(X)$ such that $\pi(Y')+\pi(Z')$ is a nontrivial element of  $TZ^{p}(X)$.

To fix notations, let $X$ be a smooth projective variety over a field $k$ of characteristic $0$ and 
let $Y \subset X$ be a subvariety of codimension $p$, with generic point $y$. Let $W \subset Y$ be a subvariety of codimension 1 in $Y $, with generic point $w$. One assumes $W$ is generically defined by $f_{1}, f_{2}, \cdots, f_{p}, f_{p+1}$ and $Y$ is generically defined by $f_{1}, f_{2}, \cdots, f_{p}$. So one  has $O_{X,y}=(O_{X,w})_{P}$, where $P$ is the ideal $(f_{1}, f_{2}, \cdots, f_{p}) \subset O_{X,w}$. 

$Y'$ is generically given by $(f_{1}+ \varepsilon g_{1}, f_{2}+ \varepsilon g_{2}, \cdots, f_{p}+ \varepsilon g_{p})$, where $g_{1}, \cdots, g_{p} \in O_{X,y}$.  We assume $g_{2}= \cdots = g_{p}=0$ in the following. Since  $O_{X,y}=(O_{X,w})_{P}$, we write $g_{1}= \dfrac{a}{b}$, where $a,b \in O_{X,w}$ and $b \notin P$. $b$ is either in or not in the maximal ideal $(f_{1}, f_{2}, \cdots, f_{p}, f_{p+1}) \subset O_{X,w}$.
\begin{lemma}  \label{lemma: trivialdeform}
If $b \notin (f_{1}, f_{2}, \cdots, f_{p}, f_{p+1})$, then $\partial_{1}^{p,-p}(\pi(Y'))=0$.
\end{lemma}

\begin{proof}
If $b \notin (f_{1}, f_{2}, \cdots, f_{p}, f_{p+1})$, then $b$ is a unit in $O_{X,w}$, this says $g_{1}= \dfrac{a}{b} \in O_{X,w}$.
Then $\pi(Y')$ is represented by the following diagram
\[
\begin{cases}
 \begin{CD}
   F_{\bullet}(f_{1}, f_{2},\cdots, f_{p}) @>>> (O_{X,w})_{P}/(f_{1}, f_{2}, \cdots,  f_{p}) \\
  F_{p}(\cong (O_{X,w})_{P}) @>  g_{1}df_{2} \wedge \cdots \wedge df_{p}>> F_{0} \otimes \Omega_{(O_{X,w})_{P}/ \mathbb{Q}}^{p-1}(\cong\Omega_{(O_{X,w})_{P}/ \mathbb{Q}}^{p-1}).
 \end{CD}
 \end{cases}
\]
Here, $F_{\bullet}(f_{1}, f_{2},\cdots, f_{p})$ is of the form
\[
 \begin{CD}
  0 @>>> F_{p} @>A_{p}>> F_{p-1} @>A_{p-1}>>  \dots @>A_{2}>> F_{1} @>A_{1}>> F_{0},
 \end{CD}
\]
where each $F_{i}=\bigwedge^{i}((O_{X,w})_{P})^{\oplus p}$.
Since $f_{p+1} \notin P$, $f_{p+1}^{-1}$ exists in $(O_{X,w})_{P}$, we can write $g_{1}df_{2} \wedge \cdots \wedge df_{p}= \dfrac{g_{1} f_{p+1}}{f_{p+1}}df_{2} \wedge \cdots \wedge df_{p}$.

$\partial_{1}^{p,-p}(\pi(Y'))$ is represented by the following diagram
\[
\begin{cases}
 \begin{CD}
   F_{\bullet}(f_{1}, f_{2},\cdots, f_{p}, f_{p+1}) @>>> O_{X, w}/(f_{1}, f_{2}, \cdots,  f_{p}, f_{p+1}) \\
  F_{p+1}(\cong O_{X,w}) @>  g_{1} f_{p+1}df_{2} \wedge \cdots \wedge df_{p}>> F_{0} \otimes \Omega_{O_{X,w}/ \mathbb{Q}}^{p-1}(\cong\Omega_{O_{X,w}/ \mathbb{Q}}^{p-1}).
 \end{CD}
 \end{cases}
\]
The complex $F_{\bullet}(f_{1}, f_{2},\cdots, f_{p}, f_{p+1})$ is of the form
\[
 \begin{CD}
  0 @>>>  \bigwedge^{p+1}(O_{X, w})^{\oplus p+1} @>A_{p+1}>> \bigwedge^{p}(O_{X, w})^{\oplus p+1} @>>>  \cdots.
 \end{CD}
\]
 Let $\{e_{1}, \cdots, e_{p+1} \}$ be a basis of $(O_{X, w})^{\oplus p+1} $, the map $A_{p+1}$ is 
\[
 e_{1}\wedge \cdots \wedge e_{p+1}  \to \sum^{p+1}_{j=1}(-1)^{j}f_{j}e_{1}\wedge \cdots \wedge \hat{e_{j}} \wedge \cdots e_{p+1},
\]
where $\hat{e_{j}}$ means to omit the $j^{th}$ term.

Since $f_{p+1}$ appears in $A_{p+1}$, 
\[
 g_{1} f_{p+1}df_{2} \wedge \cdots \wedge df_{p} \equiv 0 \in Ext_{O_{X,w}}^{p+1}(O_{X, w}/(f_{1}, \cdots,  f_{p}, f_{p+1}), \Omega_{O_{X,w}/ \mathbb{Q}}^{p-1}),
 \]
 $\partial_{1}^{p,-p}(\pi(Y'))=0$.

\end{proof}

This lemma doesn't contradict Example ~\ref{example: trivialexample} above, where $h \in (f,g,h) \subset O_{X,x}$.

\begin{theorem} \label{theorem: mainTheorem}
For $Y' \in \mathrm{T_{Y}Hilb}^{p}(X)$ which is generically defined by $(f_{1}+ \varepsilon g_{1}, f_{2}, \cdots, f_{p})$, where $g_{1}= \dfrac{a}{b} \in O_{X,y}=(O_{X,w})_{P}$, 

\begin{itemize}
\item Case 1: if $b \notin (f_{1}, f_{2}, \cdots, f_{p}, f_{p+1})$, then $\pi(Y') \in TZ^{p}(X)$, i.e., $\partial_{1}^{p,-p}(\pi(Y'))=0$. \\

\item Case 2: if $b \in (f_{1}, f_{2}, \cdots, f_{p}, f_{p+1})$, 
there exists $Z \subset X$ of codimension $p$ and exists $Z' \in \mathrm{T_{Z}Hilb}^{p}(X)$ such that $\pi(Y')+\pi(Z') \in TZ^{p}(X)$, i.e., $\partial_{1}^{p,-p}(\pi(Y')+\pi(Z'))=0$.

\end{itemize}

\end{theorem}

\begin{proof}
Case 1 is Lemma ~\ref{lemma: trivialdeform}.
Now, we consider the case $b \in (f_{1}, f_{2}, \cdots, f_{p}, f_{p+1})$. Since $b \notin (f_{1}, f_{2}, \cdots, f_{p})$,
we can write  $b= \sum_{i=1}^{p} a_{i}f_{i}^{n_{i}} + a_{p+1}f_{p+1}^{n_{p+1}}$, where $a_{p+1}$ is a unit in 
$O_{X,w}$ and each $n_{j}$ is some integer. For simplicity, we assume each $n_{j} =1$ and $a_{p+1}=1$.

Since $Y'$ is generically given by $(f_{1}+ \varepsilon g_{1}, f_{2}, \cdots, f_{p})$, then $\pi(Y')$ is represented by the following diagram( $g_{1}= \dfrac{a}{b}$):
\[
\begin{cases}
 \begin{CD}
   F_{\bullet}(f_{1}, f_{2},\cdots, f_{p}) @>>> (O_{X,w})_{P}/(f_{1}, f_{2}, \cdots,  f_{p}) \\
  F_{p}(\cong (O_{X,w})_{P}) @> \dfrac{a}{b} df_{2} \wedge \cdots \wedge df_{p}>> F_{0} \otimes \Omega_{(O_{X,w})_{P}/ \mathbb{Q}}^{p-1}(\cong\Omega_{(O_{X,w})_{P}/ \mathbb{Q}}^{p-1}).
 \end{CD}
 \end{cases}
\]
 Here, $F_{\bullet}(f_{1}, f_{2},\cdots, f_{p})$ is of the form
\[
 \begin{CD}
  0 @>>> F_{p} @>A_{p}>> F_{p-1} @>A_{p-1}>>  \dots @>A_{2}>> F_{1} @>A_{1}>> F_{0},
 \end{CD}
\]
where each $F_{i}=\bigwedge^{i}((O_{X,w})_{P})^{\oplus p}$.
Let $\{e_{1}, \cdots, e_{p} \}$ be a basis of $(O_{X, w})^{\oplus p} $, the map $A_{p}$ is 
\[
 e_{1}\wedge \cdots \wedge e_{p}  \to \sum^{p}_{j=1}(-1)^{j}f_{j}e_{1}\wedge \cdots \wedge \hat{e_{j}} \wedge \cdots e_{p},
\]
where $\hat{e_{j}}$ means to omit the $j^{th}$ term.
 
Noting $\dfrac{1}{b}-\dfrac{1}{f_{p+1}}=\dfrac{-\sum_{i=1}^{p}a_{i}f_{i}}{bf_{p+1}}$ and each $f_{i}$($i=1, \cdots, p$) appears in $A_{p}$, the above diagram representing $\pi(Y')$ can be replaced by the following one:
\[
\begin{cases}
 \begin{CD}
   F_{\bullet}(f_{1}, f_{2},\cdots, f_{p}) @>>> (O_{X,w})_{P}/(f_{1}, f_{2}, \cdots,  f_{p}) \\
  F_{p}(\cong (O_{X,w})_{P}) @> \dfrac{a}{f_{p+1}} df_{2} \wedge \cdots \wedge df_{p}>> F_{0} \otimes \Omega_{(O_{X,w})_{P}/ \mathbb{Q}}^{p-1}(\cong\Omega_{(O_{X,w})_{P}/ \mathbb{Q}}^{p-1}).
 \end{CD}
 \end{cases}
\]
Then $\partial_{1}^{p,-p}(\pi(Y'))$ is represented by the following diagram:
\[
\begin{cases}
 \begin{CD}
   F_{\bullet}(f_{1}, f_{2},\cdots, f_{p}, f_{p+1}) @>>> O_{X,w}/(f_{1}, f_{2}, \cdots,  f_{p}, f_{p+1}) \\
  F_{p+1}(\cong O_{X,w}) @> a df_{2} \wedge \cdots \wedge df_{p}>> F_{0} \otimes \Omega_{O_{X,w}/ \mathbb{Q}}^{p-1}(\cong\Omega_{O_{X,w}/ \mathbb{Q}}^{p-1}).
 \end{CD}
 \end{cases}
\]

Let $P'$ denote the prime $(f_{p+1}, f_{2}, \cdots, f_{p}) \subset O_{X,w}$ ,
then $P'$ defines a generic point $z \in X^{(p)}$ and one has $O_{X,z} = (O_{X,w})_{P'}$.
We define the subscheme
\[
Z := \overline{\{ z \}}.
\]

Let $Z'$ be a first order infinitesimal deformation of $Z$, which is generically given by $(f_{p+1}+ \varepsilon \dfrac{a}{f_{1}}, f_{2}, \cdots, f_{p})$.
$\pi(Z')$ is represented by the following diagram:
\[
\begin{cases}
 \begin{CD}
   F_{\bullet}(f_{p+1}, f_{2},\cdots, f_{p}) @>>> (O_{X, w})_{P'}/(f_{p+1}, f_{2}, \cdots,  f_{p}) \\
  F_{p}(\cong (O_{X,w})_{P'}) @> \dfrac{a}{f_{1}} df_{2} \wedge \cdots \wedge df_{p}>> F_{0} \otimes \Omega_{(O_{X,w})_{P'}/ \mathbb{Q}}^{p-1}(\cong\Omega_{(O_{X,w})_{P'}/ \mathbb{Q}}^{p-1}),
 \end{CD}
 \end{cases}
\]
and $\partial_{1}^{p,-p}(\pi(Z'))$  is represented by the following diagram:
\[
\begin{cases}
 \begin{CD}
   F_{\bullet}(f_{p+1}, f_{2},\cdots, f_{p}, f_{1}) @>>> O_{X, w}/(f_{p+1}, f_{2}, \cdots,  f_{p}, f_{1}) \\
  F_{p+1}(\cong O_{X,w}) @> a df_{2} \wedge \cdots \wedge df_{p}>> F_{0} \otimes \Omega_{O_{X,w}/ \mathbb{Q}}^{p-1}(\cong\Omega_{O_{X,w}/ \mathbb{Q}}^{p-1}).
 \end{CD}
 \end{cases}
\]
Here, $F_{\bullet}(f_{1}, f_{2},\cdots, f_{p}, f_{p+1}) $ and $F_{\bullet}(f_{p+1}, f_{2},\cdots, f_{p}, f_{1}) $ are Koszul resolutions of $O_{X, w}/(f_{1}, f_{2}, \cdots,  f_{p}, f_{p+1})$ and $O_{X, w}/(f_{p+1}, f_{2}, \cdots,  f_{p},f_{1})$ respectively.

These two Koszul complexes $F_{\bullet}(f_{1}, f_{2},\cdots, f_{p}, f_{p+1}) $ and $F_{\bullet}(f_{p+1}, f_{2},\cdots, f_{p}, f_{1}) $ are related by the following commutative diagram, see page 691 of \cite{GH},
\[
\begin{CD}
O_{X,w} @>D_{p+1}>> \wedge^{p}O_{X,w}^{\oplus p+1} @>D_{p}>> \cdots @>>> O_{X,w}^{\oplus p+1} @>D_{1}>> O_{X,w} \\
   @V\mathrm{det}A_{1}VV @V\wedge^{p}A_{1}VV @VVV  @VA_{1}VV @V=VV \\
  O_{X,w} @>E_{p+1}>> \wedge^{p}O_{X,w}^{\oplus p+1} @>E_{p}>> \cdots @>>> O_{X,w}^{\oplus p+1} @>E_{1}>> O_{X,w},
 \end{CD}
\]
where each $D_{i}$ and $E_{i}$ are defined as usual. In particular, $D_{1}=(f_{1}, f_{2},\cdots, f_{p}, f_{p+1})$, 
$E_{1}=(f_{p+1}, f_{2},\cdots, f_{p}, f_{1})$,  and $A_{1}$ is the matrix:
\[
  \left(\begin{array}{cccc}
0 & 0& 0& \cdots 1\\
0 & 1 & 0 & \cdots 0 \\
0 & 0 & 1 & \cdots 0 \\
\hdotsfor{4} \\
1 & 0 & 0 & \cdots 0
  \end{array}\right).
\]

Since $\mathrm{det}A_{1}=-1$, one has
\[
\partial_{1}^{p,-p}(\pi(Z')) = - \partial_{1}^{p,-p}(\pi(Y'))  \in Ext_{O_{X,w}}^{p+1}(O_{X, w}/(f_{1}, f_{2}, \cdots,  f_{p}, f_{p+1}), \Omega_{O_{X,w}/ \mathbb{Q}}^{p-1}),
\]
consequently, $\partial_{1}^{p,-p}(\pi(Z')+\pi(Y'))=0 \in H^{p+1}_{w}(\Omega_{O_{X,w}/ \mathbb{Q}}^{p-1})$. In other words, 
\[
\pi(Z')+\pi(Y') \in TZ^{p}(X).
\]
\end{proof}

There exists the following commutative diagram, 
which is part of the commutative diagram of  Theorem 3.14 in \cite{Y-2}(taking j=1)
\[
\begin{CD}
 \bigoplus\limits_{x \in X^{(p)}}H_{x}^{p}(\Omega_{X/\mathbb{Q}}^{p-1}) @<\mathrm{Ch}<< \bigoplus\limits_{x[\varepsilon] \in X[\varepsilon] ^{(p)}}K^{M}_{0}(O_{X,x}[\varepsilon] \ \mathrm{on} \ x[\varepsilon]) \\
     @V\partial_{1}^{p,-p}VV @Vd_{1,X[\varepsilon]}^{p,-p}VV \\
     \bigoplus\limits_{x \in X^{(p+1)}}H_{x}^{p+1}(\Omega_{X/\mathbb{Q}}^{p-1}) @<\mathrm{Ch}<\cong< \bigoplus\limits_{x[\varepsilon] \in X[\varepsilon] ^{(p+1)}}K^{M}_{-1}(O_{X,x}[\varepsilon] \ \mathrm{on} \ x[\varepsilon]).
      \end{CD}
      \]

For $Y' \in \mathrm{T_{Y}Hilb}^{p}(X)$ which is generically defined by $(f_{1}+ \varepsilon g_{1}, f_{2}, \cdots, f_{p})$, where $g_{1}= \dfrac{a}{b} \in O_{X,y}=(O_{X,w})_{P}$, we use $F_{\bullet}(f_{1}+ \varepsilon g_{1}, f_{2}, \cdots, f_{p})$ to denote the Koszul complex associated to $f_{1}+ \varepsilon g_{1}, f_{2}, \cdots, f_{p}$. Theorem \ref{theorem: mainTheorem} implies the following.

 Case 1: if $b \notin (f_{1}, f_{2}, \cdots, f_{p}, f_{p+1})$, $\partial_{1}^{p,-p}(\pi(Y'))=0$. The commutative diagram
 \[
\begin{CD}
 \pi(Y') @<\mathrm{Ch}<<   F_{\bullet}(f_{1}+ \varepsilon g_{1}, f_{2}, \cdots, f_{p})\\
     @V\partial_{1}^{p,-p}VV @Vd_{1,X[\varepsilon]}^{p,-p}VV \\
    0  @< \mathrm{Ch}<\cong <  d_{1,X[\varepsilon]}^{p,-p}(F_{\bullet}(f_{1}+ \varepsilon g_{1}, f_{2}, \cdots, f_{p})).
      \end{CD}
      \]
 says $d_{1,X[\varepsilon]}^{p,-p}(F_{\bullet}(f_{1}+ \varepsilon g_{1}, f_{2}, \cdots, f_{p})) =0 $.
 
 Case 2: if $b \in (f_{1}, f_{2}, \cdots, f_{p}, f_{p+1})$, we are reduced to considering $b=f_{p+1}$. Then there exists $Z \subset X$ which is generically defined by $(f_{p+1}, f_{2}, \cdots, f_{p})$ and exists $Z' \in \mathrm{T_{Z}Hilb}^{p}(X)$ which is generically defined by $(f_{p+1}+ \varepsilon \dfrac{a}{f_{1}}, f_{2}, \cdots, f_{p})$
such that $\partial_{1}^{p,-p}(\pi(Y')+\pi(Z'))=0$. We use $F_{\bullet}(f_{p+1}+ \varepsilon \dfrac{a}{f_{1}}, f_{2}, \cdots, f_{p})$ to denote the Koszul complex associated to $f_{p+1}+ \varepsilon \dfrac{a}{f_{1}}, f_{2}, \cdots, f_{p}$.

The commutative diagram
 \[
\begin{CD}
 \pi(Y') + \pi(Z') @<\mathrm{Ch}<<   F_{\bullet}(f_{1}+ \varepsilon \dfrac{a}{f_{p+1}}, f_{2}, \cdots, f_{p}) + F_{\bullet}(f_{p+1}+ \varepsilon \dfrac{a}{f_{1}}, f_{2}, \cdots, f_{p})\\
     @V\partial_{1}^{p,-p}VV @Vd_{1,X[\varepsilon]}^{p,-p}VV \\
    0  @< \mathrm{Ch}<\cong <  d_{1,X[\varepsilon]}^{p,-p}(F_{\bullet}(f_{1}+ \varepsilon \dfrac{a}{f_{p+1}}, f_{2}, \cdots, f_{p})+F_{\bullet}(f_{p+1}+ \varepsilon \dfrac{a}{f_{1}}, f_{2}, \cdots, f_{p})).
      \end{CD}
      \]
 says $d_{1,X[\varepsilon]}^{p,-p}(F_{\bullet}(f_{1}+ \varepsilon \dfrac{a}{f_{p+1}}, f_{2}, \cdots, f_{p})+F_{\bullet}(f_{p+1}+ \varepsilon \dfrac{a}{f_{1}}, f_{2}, \cdots, f_{p})) =0 $.
 
Recall that, in Definition 3.4 and Corollary 3.15 in \cite{Y-2}, the $p$-th Milnor K-theoretic cycles is defined as
\[
Z^{M}_{p}(D^{\mathrm{Perf}}(X[\varepsilon])) := \mathrm{Ker}(d_{1,X[\varepsilon]}^{p,-p}).
\]
 
 The above can be summarized as:
 \begin{theorem} \label{theorem: MainThmCycles}
 For $Y' \in \mathrm{T_{Y}Hilb}^{p}(X)$ which is generically defined by $(f_{1}+ \varepsilon g_{1}, f_{2}, \cdots, f_{p})$, where $g_{1}= \dfrac{a}{b} \in O_{X,y}=(O_{X,w})_{P}$, we use $F_{\bullet}(f_{1}+ \varepsilon g_{1}, f_{2}, \cdots, f_{p})$ to denote the Koszul complex associated to $f_{1}+ \varepsilon g_{1}, f_{2}, \cdots, f_{p}$.
 
\begin{itemize}
\item Case 1: if $b \notin (f_{1}, f_{2}, \cdots, f_{p}, f_{p+1})$,  then $F_{\bullet}(f_{1}+ \varepsilon g_{1}, f_{2}, \cdots, f_{p}) \in Z^{M}_{p}(D^{\mathrm{Perf}}(X[\varepsilon]))$. \ \

\item Case 2: if $b \in (f_{1}, f_{2}, \cdots, f_{p}, f_{p+1})$, we are reduced to considering $b=f_{p+1}$. Then there exists $Z \subset X$ which is generically defined by $(f_{p+1}, f_{2}, \cdots, f_{p})$ and there exists $Z' \in \mathrm{T_{Z}Hilb}^{p}(X)$ which is generically defined by $(f_{p+1}+ \varepsilon \dfrac{a}{f_{1}}, f_{2}, \cdots, f_{p})$
such that $F_{\bullet}(f_{1}+ \varepsilon \dfrac{a}{f_{p+1}}, f_{2}, \cdots, f_{p})+F_{\bullet}(f_{p+1}+ \varepsilon \dfrac{a}{f_{1}}, f_{2}, \cdots, f_{p}) \in Z^{M}_{p}(D^{\mathrm{Perf}}(X[\varepsilon]))$. 
 
\end{itemize}
 \end{theorem}

The existence of $Z$ and $Z' \in \mathrm{T_{Z}Hilb}^{p}(X)$ has applications in deformation of cycles, see \cite{Y-4} for a concrete example of eliminating obstructions to  deforming curves on a three-fold.

\end{document}